\documentclass[a4paper,12pt]{amsart}
\usepackage{amsmath,amsthm,amssymb}
\usepackage{hyperref}

\allowdisplaybreaks[1]

\newcommand{\ity}{\infty}

\newcommand{\C}{\mathbb{C}}

\newcommand{\N}{\mathbb{N}}

\numberwithin{equation}{section}

\newtheorem{theorem}{Theorem}[section]

\newtheorem{corollary}[theorem]{Corollary}
\theoremstyle{remark}

\thanks {The research work of the first  author is supported by research fellowship from Council of Scientific and Industrial Research (CSIR), New Delhi.}

\begin{document}
\title[dynamics and singularities]{On dynamics of composite entire functions and singularities}
\author[D. Kumar]{Dinesh Kumar}
\address{Department of Mathematics, University of Delhi,
Delhi--110 007, India}

\email{dinukumar680@gmail.com }
\author[S. Kumar]{Sanjay Kumar}

\address{Department of Mathematics, Deen Dayal Upadhyaya College, University of Delhi,
Delhi--110 007, India }

\email{sanjpant@gmail.com}

\begin{abstract}
We consider the dynamical properties of transcendental entire functions and their compositions. We give  several conditions under which Fatou set of a transcendental entire function $f$ coincide with that of $f\circ g,$ where $g$ is another transcendental entire function. We also prove some result giving relationship between singular values of transcendental entire functions and their compositions.
\end{abstract}
\keywords{Singular value, normal family, wandering domain, bounded type, permutable}

\subjclass[2010]{37F10, 30D05}

\maketitle

\section{Introduction}\label{sec1}
 Let $f$ be a transcendental entire function. For $n\in\N$ let $f^n$ denote the n-th iterate of $f.$ The set $F(f)=\{z\in\C : \{f^n\}_{n\in\N}\,\text{ is normal in some neighborhood of}\, z\}$ is called the Fatou set of $f$ or the set of normality of $f$ and its complement $J(f)$ is the Julia set of $f$.  The Fatou set is open and completely invariant: $z\in F(f)$ if and only if $f(z)\in F(f)$ and consequently $J(f)$ is completely invariant. The Julia set of a transcendental entire function is non empty, closed perfect set and unbounded. All these results and more can be found in Bergweiler \cite{berg1}. If $U$ is a component of Fatou set $F(f),$ then $f(U)$ lies in some component $V$ of $F(f)$ and  $V\setminus f(U)$ is a set which contains atmost one point, \cite{berg3}. This result was also proved in \cite{MH} independently. A component $U$ of  Fatou set of $f$ is called a wandering domain if $U_k\cap U_l=\emptyset$ for $k\neq l,$ where $U_k$ denotes the component of $F(f)$ containing $f^k(U),$ otherwise $U$ is called a preperiodic component of $F(f),$ $f^k(U_l)=U_l$ for some $k,l\geq 0.$ If $f^k(U)=U,$ for some $k\in\N,$ then $U$ is called a periodic component of $F(f).$ Sullivan \cite{Sullivan} proved that the Fatou set of any rational function has no wandering domain. It was Baker \cite{baker2} who gave the first example of an entire function having wandering domain. Thereafter  several other examples of wandering domains  have been given by various authors, see \cite{APS3}. Certain classes of transcendental functions which do not have wandering domains are also known, see ~\cite{baker3, baker7, berg1, berg2, el2, keen, stallard}.

Two functions $f$ and $g$ are called permutable if $f\circ g=g\circ f.$
Fatou  \cite{beardon}, proved that if $f$ and $g$ are two rational functions which are permutable, then $F(f)=F(g)$. This was the most famous result  that motivated the dynamics of composition of complex functions. Similar results for transcendental entire functions is still not known, though it holds in some very special cases,  \cite[Lemma 4.5]{baker3}.
If $f$ and $g$ are transcendental entire functions, then so is $f\circ g$ and $g\circ f$ and the dynamics of one composite entire function helps in the study of the dynamics of the other and vice-versa. 
If $f$ and $g$ are transcendental entire functions, the dynamics of $f\circ g$ and $g\circ f$ are very similar. In  ~\cite{berg4, Poon}, it was shown $f\circ g$ has  wandering domains if and only if $g\circ f$ has  wandering domains. In \cite{dinesh1}          the authors have constructed several  examples where the dynamical behavior of $f$ and $g$ vary greatly from the dynamical behavior of $f\circ g$ and $g\circ f.$ Using approximation theory of entire functions, the authors have shown the existence of entire functions $f$ and $g$ having infinite number of domains satisfying various properties and relating it to their composition. They explored and enlarged all the maximum possible ways of the solution in comparison to the past result worked out.

Recall that $w\in\C$ is a critical value of a transcendental entire function $f$ if there exist some $w_0\in\C$ with $f(w_0)=w$ and $f'(w_0)=0.$ Here $w_0$ is called a critical point of $f.$ The image of a critical point of $f$ is  critical value of $f.$ Also recall that $\zeta\in\C$ is an asymptotic value of a transcendental entire function $f$ if there exist a curve $\Gamma$ tending to infinity such that $f(z)\to \zeta$ as $z\to\ity$ along $\Gamma.$ The set of all critical points, critical values and asymptotic values of $f$ will be denoted by $CP(f), CV(f)$ and $AV(f)$ respectively. Recall the Eremenko-Lyubich class
 \[\mathcal{B}=\{f:\C\to\C\,\,\text{transcendental entire}: \text{Sing}{(f^{-1})}\,\text{is bounded}\},\]
(where Sing$f^{-1}$ is the set of critical values and asymptotic values of $f$ and their finite limit points). Each $f\in\mathcal{B}$ is said to be of bounded type. A transcendental entire function $f$ is of finite type if Sing$f^{-1}$ is a finite set. Furthermore, if the transcendental entire functions $f$ and $g$ are of bounded type then so is $f\circ g$ as Sing $((f\circ g)^{-1})\subset$ Sing $f^{-1}\cup f(\text{Sing}(g^{-1})),$ \cite{berg4}. Singularities of a transcendental map plays an important role in its dynamics. For any transcendental entire function Sing$f^{-1}\neq\emptyset,$ \cite[p.\ 66]{Hua}. It is well known, \cite{el2, keen}, if $f$ is of finite type then it has no wandering domains. Recently Bishop \cite{bishop} has constructed an example of a function of bounded type having a wandering domain. Let $E(f)=\cup_{n\geq 0}f^{n}(\text{Sing}f^{-1})$ and $E'(f)$ be the derived set of $E(f),$ that is, the set of finite limit points of $E(f).$ It is well known  \cite{berg2}, if $U\subset F(f)$ is a wandering domain, then all limit functions of $\{f^n\vert_{U}\}$ are constant and are contained in $(E'(f)\cap J(f))\cup\{\ity\}.$ Furthermore, if  $C$ denotes the class of transcendental entire functions with $J(f)\cap E'(f)=\emptyset,$ and if $f\in\mathcal{B}\cap C,$ then $f$ does not have any wandering domains, \cite[Corollary]{berg2}.

 Here we shall consider the relationship between Fatou sets and singular values of transcendental entire functions $f, g$ and $f\circ g$. Some of the results are motivated by the work in \cite{ap1,ap2}. We shall give various conditions under which Fatou sets of $f$ and $f\circ g$ coincide. We have also considered relation between the singular values of $f, g$ and their compositions. Moreover, the relation between singular values of $f$ and $g$ in terms of conjugating map $\phi:\C\to\C$ has also been investigated. Recall two entire functions $f$ and $g$ are conjugate if there exist a conformal map $\phi:\C\to\C$ with $\phi\circ f=g\circ\phi.$ By a conformal map $\phi:\C\to\C$ we mean an analytic and univalent map of the complex plane $\C.$
\section{Theorems and their proofs}\label{sec2}

\begin{theorem}\label{sec2,thm1}
Let $f$ and $g$ be two permutable transcendental entire functions. Then
\begin{enumerate}
\item [(i)] $F(f\circ g)$ is completely invariant under $f$ and $g$ respectively;
\item [(ii)] $F(f\circ g)\subset F(f)\cap F(g);$
\item [(iii)] For any two positive integers $i$ and $j, F(f^i\circ g^j)=F(f\circ g).$
\end{enumerate}
\end{theorem}
\begin{proof}
\begin{enumerate}
\item [(i)]Bergweiler and Wang \cite{berg4}, showed that $z\in F(f\circ g)$ if and only if $f(z)\in F(g\circ f).$ Since $f\circ g=g\circ f, F(f\circ g)$ is completely invariant under $f$ and by symmetry, under $g$ respectively. 
\item [(ii)] As $F(f\circ g)$ is completely invariant under $f$ and $g$, and so it is forward invariant under them. So $f(F(f\circ g))\subset F(f\circ g)$ and $g(F(f\circ g))\subset F(f\circ g)$ which by Montel's Normality Criterion implies $F(f\circ g)\subset F(f)$ and $F(f\circ g)\subset F(g).$ Hence $F(f\circ g)\subset F(f)\cap F(g).$
\item [(iii)] For $i, j\in\N,$ assume $i\geq j.$ Now $F(f\circ g)=F(f^i\circ g^i)\subset F(f^i\circ g^j)\subset F(f\circ g),$ which implies $F(f^i\circ g^j)=F(f\circ g)$ for all $i, j\in\N.$\qedhere
\end{enumerate}
\end{proof}
We now provide conditions under which $F(f)$ equals $F(f\circ g).$
\begin{theorem}\label{sec2,thm2}
Let $f$ and $g$ be permutable transcendental entire functions. Then the following holds:
\begin{enumerate}
\item [(i)] If $\ity$ is not a limit function of any subsequence of $\{f^n\}$ in a component of $F(f)$, then $F(f)=F(f\circ g);$
\item [(ii)]  If $f$ is of bounded type, then $F(f)=F(f\circ g);$
\item [(iii)] If $g=af+b, a, b\,\, constants, a\neq 0,$ then $F(f)=F(f\circ g)$;
\item [(iv)] If $f$ is periodic of period $c$ and $g=f+c,$ then  $F(f)=F(f\circ g).$
\end{enumerate}
\end{theorem}
\begin{proof}
\begin{enumerate}
\item [(i)]Using Baker's result \cite[Lemma 4.3\,(ii)]{baker3}, we have $g(F(f))\subset F(f).$ Also $F(f)$ is forward invariant under $f$ and hence $F(f)$ is forward invariant under $f\circ g$, which by Montel's Normality Criterion implies $F(f)\subset F(f\circ g).$ Using Theorem \ref{sec2,thm1}(ii), we get the desired result.
\item [(ii)] As $f$ is of bounded type, $g(F(f))\subset F(f),$ which implies $F(f)\subset F(g)$ by Montel's Normality Criterion. Thus $F(f)$ is forward invariant under $f, g$ and hence under $f\circ g.$ Again by Montel's Normality Criterion $F(f)\subset F(f\circ g).$ From Theorem \ref{sec2,thm1}(ii), $F(f\circ g)\subset F(f)$ and hence we have the desired result.
\item [(iii)] In Baker \cite{baker3}, it was shown that $g(F(f))\subset F(f)$ and using \ref{sec2,thm2}(i), we obtain the desired result.
\item [(iv)] Observe that $f\circ g(z)=f^2(z)$ and hence $F(f\circ g)=F(f).$\qedhere
\end{enumerate}
\end{proof}

We now show if the composite entire function and its right factor have bounded set of asymptotic values, then so does the left factor.
\begin{theorem}\label{sec2,thm4}
Let $f$ and $g$ be transcendental entire functions such that the asymptotic values of $f\circ g$ and $g$ are bounded. Then the asymptotic values of $f$ are also bounded.
\end{theorem}
\begin{proof}
Suppose the asymptotic values of $f$ are unbounded. Then  there exist a sequence $\{z_n\}\subset AV(f)$ such that for each $n\in\N,$ $|z_n|>n$. Corresponding to each $z_n$ there exist a curve $\Gamma_n$ tending to $\ity,$ with $f(z)\to z_n$ as $z\to\ity$ along $\Gamma_n.$ Let $\Gamma_n'$ be an analytic branch of $g^{-1}(\Gamma_n).$ Then $\Gamma_n'$ must tend to $\ity$, for if $\Gamma_n'$ tends to some finite limit say $z_0,$ then $z_0$ must either be a pole or an essential singularity of $g$ which contradicts to $g$ being an entire function. Now as $z\to\ity$ on $\Gamma_n', (f\circ g)(z)\to z_n$ which contradicts the hypothesis that $AV(f\circ g)$ is bounded and hence the result.
\end{proof}
We now study  relation between singular values of two conjugate entire functions:
\begin{theorem}\label{sec2,thm5}
Let $f$ and $g$  be transcendental entire functions which are conjugate under a conformal map $\phi:\C\to\C,$ that is, $\phi\circ f=g\circ\phi,$ then
\begin{enumerate}
\item [(i)] $AV(g)=\phi(AV(f));$
\item [(ii)] $\phi(CP(f))\subset CP(g);$
\item [(iii)] $\phi(CV(f))\subset CV(g);$
\item [(iv)] $CP(f)\subset CP(\phi\circ f);$
\item [(v)] If $\phi$ maps $CP(g)$ inside $CP(g),$ then $CP(g)\subset CP(f).$
\end{enumerate}
\end{theorem}
\begin{proof}
\begin{enumerate}
\item [(i)] We first show $AV(g)\subset \phi(AV(f)).$ Let $z_0\in AV(g).$ Then there exist a curve $\Gamma$ tending to $\ity$ such that $g(z)\to z_0$ on $\Gamma,$ that is, $\phi\circ f\circ\phi^{-1}(z)\to z_0$ on $\Gamma.$ As $z\to\ity$ on $\Gamma,  \phi^{-1}(z)$ also tends to $\ity$ on $\Gamma$ which implies $z_0$ is an asymptotic value of $\phi\circ f.$ Thus $z_0\in AV(\phi\circ f)\subset AV(\phi)\cup\phi(AV(f))$, \cite{berg4}. As $AV(\phi)=\emptyset,$ we have $AV(g)\subset\phi(AV(f)).$ For the backward implication let $w_0\in AV(f).$ Then there exist a curve $\Gamma'$ tending to $\ity$ such that $f(z)\to w_0$ on $\Gamma'.$ As $\phi$ is continuous, $\phi\circ f(z)\to\phi(w_0)$ on $\Gamma'$ which implies $\phi(w_0)$ is an asymptotic value of $\phi\circ f=g\circ\phi.$ Hence  $\phi(AV(f))\subset AV(g\circ\phi)\subset AV(g)\cup g(AV(\phi))$ and as $AV(\phi)=\emptyset,$ we have $\phi(AV(f))\subset AV(g)$ and hence the result.
\item [(ii)] Let $w\in\phi(CP(f)),$ then $w=\phi(z_0),$ for some $z_0\in CP(f),$   $f'(z_0)=0.$ Now $(g\circ\phi)'(z_0)=(\phi\circ f)'(z_0)$ and since $\phi'(z_0)\neq 0$ we have $\phi(z_0)\in CP(g)$, that is, $w\in CP(g)$ and hence $\phi(CP(f))\subset CP(g).$
\item [(iii)] As shown above we have, $\phi(z_0)\in CP(g)$ and so $g\circ\phi(z_0)\in CV(g),$ that is,  $\phi\circ f(z_0)\in CV(g).$ As $z_0\in CP(f)$ we have  $f(z_0)\in CV(f)$ and hence $\phi(CV(f))\subset CV(g)$.
\item [(iv)] If $z_0\in CP(f),$ we have $f'(z_0)=0.$ Then $(\phi\circ f)'(z_0)=\phi'(f(z_0))f'(z_0)=0$ which implies $z_0\in CP(\phi\circ f)$ and hence the result.
\item [(v)] Let $w_0\in CP(g).$ Then by hypothesis $\phi(w_0)\in CP(g).$ Now $0=(g\circ\phi)'(w_0)=(\phi\circ f)'(w_0)=\phi'(f(w_0))f'(w_0)$ and since $\phi'(f(w_0))\neq 0,$ we have $f'(w_0)=0$ which implies the result.\qedhere
\end{enumerate}
\end{proof}
An immediate consequence of above theorem is 
\begin{corollary}\label{sec2,cor1}
Let $f, g$ and $\phi$ be as in the previous theorem. If $AV(f)$ is bounded, then so is $AV(g).$
\end{corollary}
\begin{theorem}\label{sec2,thm6}
Let $f$ and $g$ be transcendental entire functions. Then
\begin{enumerate}
\item\ $f(AV(g))\subset AV(f\circ g);$
\item\ $g(CP(f\circ g))\subset CP(f)\cup CV(g).$
\end{enumerate}
\end{theorem}
\begin{proof}
\begin{enumerate}
\item\ Let $z_0\in AV(g),$ then there exist a curve $\Gamma$ tending to $\ity$ with $g(z)\to z_0$ as $z\to\ity$ on $\Gamma$ and so $f\circ g(z)\to f(z_0)$ on $\Gamma$ 
which implies $f(z_0)\in AV(f\circ g)$ and hence $f(AV(g))\subset AV(f\circ g).$
\item\ Let $w\in g(CP(f\circ g)),$ then $w=g(z_0),$ for some $z_0\in CP(f\circ g).$ Now $0=(f\circ g)'(z_0)=f'(g(z_0))g'(z_0),$ which implies either $g(z_0)\in CP(f)$ or $z_0\in CP(g)$ and so $g(z_0)\in CP(f)\cup CV(g)).$ Hence $g(CP(f\circ g))\subset CP(f)\cup CV(g).$\qedhere
\end{enumerate}
\end{proof}
In the next result it is shown that if the critical points of composite entire function are bounded, then so is that of its right and left factors.
\begin{theorem}\label{sec2,thm7}
Let $f$ and $g$ be  transcendental entire functions such that $CP(g\circ f)$ is bounded. Then $CP(f)$ is also bounded and $g$ can have atmost one critical point which is a Picard exceptional value of $f.$
\end{theorem}
\begin{proof}
If $z_0\in CP(f),$ then $f'(z_0)=0.$ Now $(g\circ f)'(z_0)=g'(f(z_0))f'(z_0)=0,$ and so $z_0\in CP(g\circ f)$ which implies $CP(f)\subset CP(g\circ f)$ and hence $CP(f)$ is bounded. Now suppose $w\in CP(g)$  is not a Picard exceptional value for $f.$ Then there exist infinitely many $w_n$ such that $f(w_n)=w.$ Now $(g\circ f)'(w_n)=g'(f(w_n))f'(w_n)=g'(w)f'(w_n)=0,$ and so $w_n\in CP(g\circ f)$. Thus $\{w_n\}\subset CP(g\circ f)$ is an infinite bounded set and hence has a limit point by Bolzano Weierstrass Theorem, which contradicts the fact that zero set of an analytic function are isolated. Hence $w$ is a Picard exceptional value for $f$ and the result follows.
\end{proof}
We next show if the critical points of composite entire function are bounded, then so is that of its right factor and moreover the left factor can have atmost one critical value.
\begin{theorem}\label{sec2,thm9}
Let $f$ and $g$ be transcendental entire functions such that critical points of $f\circ g$ are bounded. Then critical values of $g$ are also bounded. Moreover if $f$ permutes with $f'$ and $0$ is a fixed point of $f,$ then $f$ can have atmost one critical value which is a Picard exceptional value for $g.$
\end{theorem}
\begin{proof}
Let $w_0\in CV(g).$ Then $w_0=g(z_0)$ for some $z_0\in CP(g), g'(z_0)=0.$ Now $(f\circ g)'(z_0)=f'(g(z_0))g'(z_0)=0,$ and so $z_0\in CP(f\circ g),$ which implies  $g(z_0)\in g(CP(f\circ g)),$ that is, $w_0\in g(CP(f\circ g)).$ As $CP(f\circ g)$ is bounded, $g(CP(f\circ g))$ is also bounded and hence $CV(g)$ is bounded.
Now let $w\in CV(f)$ and suppose $w$ is not a Picard exceptional value for $g.$ Then there are infinitely many $\{w_n\}$ such that $g(w_n)=w.$ We have $w=f(z')$ for some $z'\in CP(f), f'(z')=0.$ Thus $(f\circ g)'(w_n)=f'(g(w_n))g'(w_n)=f'(f(z'))g'(w_n),$ and since $f'\circ f=f\circ f'$ we obtain $(f\circ g)'(w_n)=0$ which implies $w_n\in CP(f\circ g).$ Thus $\{w_n\}\subset CP(f\circ g)$ is an infinite bounded set and hence has a limit point by Bolzano Weierstrass Theorem, which contradicts the fact that zero set of an analytic function are isolated. Hence $w$ has to be a Picard exceptional value for $g$. Since Picard exceptional value for a transcendental entire function can be atmost one, $f$ can have atmost one critical value. 
\end{proof}
The next theorem deals with critical values and asymptotic values of entire functions and their compositions.
\begin{theorem}\label{sec2,thm8}
Let $f$ and $g$ be transcendental entire functions. Then 
\begin{enumerate}
\item [(i)] $AV(f)\subset AV(f\circ g);$
\item [(ii)] $g(CV(f))\subset CV(g\circ f).$
\end{enumerate}
\end{theorem}
\begin{proof}
\begin{enumerate}
\item [(i)] Let $w\in AV(f).$ Then there exist a curve $\Gamma$ tending to $\ity$ such that $f(z)\to w$ as $z\to\ity$ on $\Gamma.$ Let $\Gamma'$ be an analytic branch of $g^{-1}(\Gamma).$ Then $\Gamma'$ must tend to $\ity,$ for if $\Gamma'$ tends to a finite limit $z_0$ say, then $z_0$ is either a pole or an essential singularity of $g$ which is impossible. 
Thus as $z\to\ity$ on this curve  $\Gamma',\, f\circ g(z)\to w,$ 
which implies $w\in AV(f\circ g)$ and hence $AV(f)\subset AV(f\circ g).$
\item [(ii)] Let $w_0\in CV(f).$ Then $w_0=f(z_0),$ for some $z_0\in CP(f).$ Now $(g\circ f)'(z_0)=g'(f(z_0))f'(z_0)=0,$ and so $z_0\in CP(g\circ f)$ which implies $g\circ f(z_0)\in CV(g\circ f),$ that is,  $g(w_0)\in CV(g\circ f)$ and hence the result.\qedhere
\end{enumerate}
\end{proof}
An immediate consequence of Theorem \ref{sec2,thm8}(i) is
\begin{corollary}\label{sec2,cor2}
If $AV(f\circ g)$ is bounded, then so is $AV(f).$
\end{corollary}


\begin{thebibliography}{00}
\bibitem{baker2} I. N. Baker, An entire function which has wandering domain, J. Austral. Math. Soc. (Ser A) \textbf{22} (1976), 173-176.
\bibitem{baker3} I. N. Baker, Wandering domains in the iteration of entire functions, Proc. London Math. Soc. \textbf{49} (1984), 563-576.
\bibitem{baker7} I. N. Baker, J. Kotus and Lu Yinian, Iterates of meromorphic functions IV: Critically finite functions, Results Math. \textbf{22} (1992), 651-656.
\bibitem{baker8} I. N. Baker and A. P. Singh, Wandering domain in the iteration of composition of entire functions, Ann. Acad. Sci. Fenn. Ser. A. I. Math. \textbf{20} (1995), 149-153.
\bibitem{beardon} A. F. Beardon, \emph{Iteration of rational functions}, Springer Verlag, (1991).
\bibitem{berg1} W. Bergweiler, Iteration of meromorphic functions, Bull. Amer. Math. Soc. \textbf{29} (1993), 151-188.
\bibitem{berg2} W. Bergweiler, M. Haruta, H. Kriete, H. G. Meier and N. Terglane, On the limit functions of iterates in wandering domain, Ann. Acad. Sci. Fenn. Ser. A. I. Math. \textbf{18} (1993), 369-375.
\bibitem{berg3} W. Bergweiler and St. Rohde, Omitted values in domains of normality, Proc. Amer. Math. Soc. \textbf{123} (1995), 1857-1858.
\bibitem{berg4} W. Bergweiler and Y. Wang, On the dynamics of composite entire functions, Ark. Math. \textbf{36} (1998), 31-39.
\bibitem{bishop} C. J. Bishop, Constructing entire functions by quasiconformal folding, preprint (2011).
\bibitem{el1}  A. E. Eremenko and M. Yu. Lyubich, Iterates of entire functions, Soviet Math. Dokl. \textbf{30} (1984), 592-594.
\bibitem{el2} A. E. Eremenko and M. Yu. Lyubich, Dynamical properties of some classes of entire functions, Ann. Inst. Fourier, Grenoble, \textbf{42} (1992), 989-1020.
\bibitem{keen} L. R. Goldberg and L. Keen, A finiteness theorem for a dynamical class of entire functions, Ergodic Theory and Dynamical Systems, \textbf{6} (1986), 183-192.
\bibitem{MH} M. E. Herring, Mapping properties of Fatou components, Ann. Acad. Sci. Fenn. Math. \textbf{23} (1998), 263-274.
\bibitem{Hua} X. H. Hua and C. C. Yang, \emph{Dynamics of transcendental functions}, Gordon and Breach Science Pub. (1998).
\bibitem{dinesh1} D. Kumar, G. Datt and S. Kumar, Dynamics of composite entire functions, arXiv:math.DS/12075930, (2013), submitted for publication.
\bibitem{Poon} K. K. Poon and C. C. Yang, Dynamics of composite entire functions, Proc. Japan. Acad. Sci. \textbf{74} (1998), 87-89.
\bibitem{APS3} A. P. Singh, On the dynamics of composition of entire functions, Math. Proc. Camb. Phil. Soc. \textbf{134} (2003), 129-138.
\bibitem{ap1} A. P. Singh and R. Sharma, On the singularities of composition of entire functions, Analysis in theory and applications, \textbf{21} (2005), 370-376.
\bibitem{ap2} A. P. Singh and Y. Wang, Julia sets of permutable holomorphic maps, Science in China Series A: Mathematics, \textbf{49} (2006), 1715-1721.
\bibitem{stallard} G. W. Stallard, A class of meromorphic functions with no wandering domains, Ann. Acad. Sci. Fenn. Ser. A. I. Math. \textbf{16} (1991), 211-226.
\bibitem{Sullivan} D. Sullivan, Quasiconformal homeomorpism and dynamics I, Solution of the Fatou Julia problem on wandering domains, Ann. Math. \textbf{122} (1985), 401-418.
\end{thebibliography}
\end{document}